\documentclass[11pt]{amsart}
\usepackage{amsmath,amssymb,bm}
\usepackage{verbatim,a4wide} 
\newtheorem{lemma}{Lemma}[section]

\newtheorem{remark}[lemma]{Remark}

\newtheorem{theorem}[lemma]{Theorem}
\newtheorem{conjecture}[lemma]{Conjecture}

\numberwithin{equation}{section}
\newcommand{\R}{\mathbb{R}}
\newcommand{\C}{\mathbb{C}}
\newcommand{\N}{\mathbb{N}}
\newcommand{\Z}{\mathbb{Z}}
\newcommand{\T}{\mathbb{T}}

\begin{document}
\title
[Invariant measures for BO   ]
{New non degenerate invariant measures for the Benjamin-Ono equation  }

\author[N. Tzvetkov]{Nikolay Tzvetkov}
\address[N. Tzvetkov]{Ecole Normale Sup\'erieure de Lyon, UMPA, UMR CNRS-ENSL 5669, 46, all\'ee d'Italie, 69364-Lyon Cedex 07, France} 
\email{nikolay.tzvetkov@ens-lyon.fr}
\begin{abstract}
We show that the recent work by  G\'erard-Kappeler-Topalov can be used in order to construct new non degenerate invariant measures for the Benjamin-Ono equation on
the Sobolev spaces $H^s$, $s>-1/2$. 
\end{abstract}
\maketitle
\section{Introduction}
Consider the Cauchy problem for the Benjamin-Ono equation, posed on the torus 
\begin{equation}\label{BO_pak}
\partial_t u=H \partial_x^2 u-\partial_x(u^2),\quad u\vert_{t=0}=u_0.
\end{equation}
In \eqref{BO_pak}, $u(t,\cdot)$ is a real valued function or a distribution on the torus $\T=\R/ 2\pi\Z$ and $t\in\R$ (the problem is time reversible). 
The anti-symmetric map $H$ is the Hilbert transform on periodic functions. The problem \eqref{BO_pak} is an infinite dimensional Hamiltonian system. 
The study of the Cauchy problem \eqref{BO_pak}  attracted a lot of attention, we refer to the very recent work \cite{monica} and the references therein. 
In this note, we are interested in invariant measures for \eqref{BO_pak}. The invariant measures problem for \eqref{BO_pak}  was addressed in 
 \cite{Deng, DTV, TV13a,TV13b,TV14}. Our aim is to show that the recent work by  G\'erard-Kappeler-Topalov \cite{GKP} can be used in order to construct a large family of new non degenerate invariant measures for \eqref{BO_pak}. In particular, we will construct such measures living in $H^s$ for every $s>-1/2$, which is a significant extension of \cite{Deng}. 
 \\
 
 A natural space for studying \eqref{BO_pak} is the Sobolev space 
 $$
 H^s_0(\T)=\{ u\,\colon\,\, \hat{u}(0)=0,\quad \hat{u}(n)=\overline{\hat{u}(-n)}\},
 $$ 
 endowed with the norm
 $$
 \sum_{n\in\Z}\, |n|^{2s} |\hat{u}(n)|^2\,.
 $$
The zero Fourier coefficient is invariant under \eqref{BO_pak} and one may introduce similar (affine) spaces in which the value of the zero Fourier coefficient is prescribed as a fixed real number. The analysis of \eqref{BO_pak} in such spaces does not present a considerable novelty and therefore we will restrict to the spaces  $H^s_0(\T)$. Thanks to the remarkable works \cite{GK,GKP}, we know that \eqref{BO_pak} is globally well-posed in $H^s_0(\T)$, $s>-1/2$, and the dynamics is almost periodic in time. It is a natural question whether  \cite{GK,GKP} can be used to deduce new results on the statistical description of the flow of \eqref{BO_pak}. Our goal here is to make a first step in this direction. 
\begin{theorem}\label{thm2}
Let $s>-1/2$.  Then there is a probability measure $\rho_s$ on $H^s_0(\T)$  which is invariant under the flow of the Benjamin-Ono equation, defined in \cite{GK,GKP}, and satisfies :
\begin{enumerate}
\item
 $\rho_s(H^\sigma_0(\T))=0$ for every $\sigma>s$.
 \item
 Every set of full  $\rho_s$  measure is dense in $H^s_0(\T)$. 
 \end{enumerate}
\end{theorem}
In order to prove Theorem \ref{thm2}, we construct a large family of invariant measures for the Benjamin-Ono equation written in the Birkhoff coordinates, introduced in \cite{GK,GKP}.
In order to prove the invariance properties in the Birkhoff coordinates side, we simply exploit the Hamiltonian structure and some simple rotation invariance properties.  Then we use some basic properties of the Birkhoff map proved in  \cite{GK,GKP}. The analysis below gives more precise informations on the possible measures  $\rho_s$. However, many basic questions remain open, see Section~3 below. 
\\

The rest of the paper is organized as follows.  In the next section we study invariant measures for the Benjamin-Ono equation written in the Birkhoff coordinates. In Section~3 we complete the proof of  Theorem~\ref{thm2} and we make a conjecture concerning the nature of the obtained invariant measures. Finally, in the last section, we introduce a renormalized flow for 
 the Benjamin-Ono equation written in the Birkhoff coordinates for random data going beyond the analysis of  \cite{GK,GKP}. It is a challenging open problem to suitably extend the Birkhoff map of \cite{GK,GKP} to this setting. This would in particular allow to solve (after a renormalization) the Benjamin-Ono equation with white noise initial data.  See also Remark~\ref{rkrkrk} below.
\section{Analysis in  Birkhoff coordinates
}
\subsection{The flows}
By definition $\N=\{1,2,\cdots\}$. 
We denote by $h^s$ the set of sequences of complex numbers $\zeta=(\zeta_n)_{n\in \N}$ such that
\begin{equation}\label{norm}
\big(\sum_{n\in\N} |n|^{2s}\, |\zeta_n|^2\big)^{\frac 1 2}<\infty\,.
\end{equation}
We endow $h^s$ with the norm \eqref{norm} resulting from the natural scalar product.  
For $N\geq 1$, we denote by $E_N$ the space spanned by the first $N$ vectors in the canonical base of $h^s$ and by $\pi_N:h^s\rightarrow E_N$ the corresponding orthogonal projection.  Therefore $E_N=\{u\in h^s\colon u=\pi_N u \}$.
\\

When written in Birkhoff coordinates the Benjamin-Ono equation  in $H_0^{s-\frac 1 2}(\T)$  (cf. \cite{GK,GKP}) becomes the following equation in $h^s$ 
\begin{equation}\label{Birk}
\partial_t \zeta_n=i\big(
n^2-2\sum_{k=1}^\infty \min(n,k) |\zeta_k|^2
\big)\zeta_n, \quad n\geq 1. 
\end{equation}
For $s\geq 0$ the solutions of \eqref{Birk} are given by
\begin{equation}\label{Birk_sol}
\zeta_n(t)=\zeta_n(0)\, \exp\Big(it\big(n^2-2\sum_{k=1}^\infty \min(k,n)|\zeta_k(0)|^2\big)\Big)\,\,.
\end{equation}
We readily see that for $(\zeta_n(0))_{n\in\N}\in h^s$, $s\geq 0$ the expression \eqref{Birk_sol} is well defined while if  $(\zeta_n(0))_{n\in\N}\notin h^0$ the expression \eqref{Birk_sol} is not well defined.  We denote by $\Phi(t)$ the flow map of \eqref{Birk} defined on $h^s$, $s\geq 0$.
\\

For $N\geq 1$, we consider the following finite dimensional truncation of \eqref{Birk}, posed on $E_N$,
\begin{equation}\label{Birk_N}
\partial_t \zeta_n=i\big(n^2-2\sum_{k=1}^N \min(n,k) |\zeta_k|^2\big)\zeta_n,\quad 1\leq n\leq N, \,\, \zeta_n(t)=0, n>N.
\end{equation}
If $\zeta_n=\xi_n+i\eta_n$ then \eqref{Birk_N} becomes 
\begin{equation}\label{Hamilton}
\partial_{t}\xi_n=\frac{\partial H}{\partial {\eta_n}},\quad \partial_{t}\eta_n=-\frac{\partial H}{\partial {\xi_n}},\,1\leq n\leq N,\, \zeta_n(t)=0,n>N,
\end{equation}
where the Hamiltonian $H$ is given by
$$
H(\xi_1,\cdots,\xi_N,\eta_1,\cdots,\eta_N)=-\frac{1}{2}\sum_{k=1}^N  k^2 (\xi_k^2+\eta_k^2)+\frac{1}{2}
\sum_{k=1}^N \Big(\sum_{k_1=k}^N (\xi_{k_1}^2+\eta_{k_1}^2)\Big)^2\,.
$$
The solutions of \eqref{Birk_N} are given by
\begin{equation}\label{Birk_sol_N}
\zeta_n(t)=\zeta_n(0)\, \exp\Big(it\big(n^2-2\sum_{k=1}^N \min(k,n)|\zeta_k(0)|^2\big)\Big)\,\,,\, 1\leq n\leq N, \,\, \zeta_n(t)=0, n>N.
\end{equation}
Let us denote by $\Phi_N(t)$ the flow map of \eqref{Birk_N} defined on $E_N$. We have that  $\Phi_N(t)$ is the restriction of $\Phi(t)$ to $E_N$.
Using \eqref{Birk_sol}, \eqref{Birk_sol_N} and the dominated convergence,  we obtain that  
\begin{equation}\label{converge}
\lim_{N\rightarrow \infty}\, \|\Phi(t)(\zeta)-\Phi_N(t)(\pi_N \zeta)\|_{h^s}=0, \quad \forall\, \zeta\in h^s,\, s\geq 0.
\end{equation}
\subsection{The measures}
Let $\theta$ be a probability measure on $\C\simeq \R^2$,  defined by
\begin{equation}\label{theta}
d\theta=f(x,y)dxdy\,,
\end{equation}
where $f$ is supposed to be a {\it radial} function, i.e.  $f(x,y)$ only depends on $x^2+y^2$.
Therefore $\theta$ is invariant under the action of the rotations of $\R^2$, a property which will be crucially used in the sequel. 
We suppose that the support of $f$ is $\R^2$.  Finally, we also suppose that 
\begin{equation}\label{moment2}
\int_{\R^2}(x^2+y^2)f(x,y)dxdy<\infty\,.
\end{equation}
The most typical example of an admissible $f$ is the gaussian 
$$
f(x,y)=\frac{1}{\pi}e^{-(x^2+y^2)}\,.
$$
Let $s\geq 0$ and let us fix a sequence $(\zeta^{*}_n)_{n\in\N}\in h^s$.
Let $(g_n(\omega))_{n\in\N}$ be a sequence of independent, identically distributed complex valued random variables with law $\theta$.
Suppose that $g_n$ are defined on a probability space $(\Omega,{\mathcal F},p)$.  Let $\mu$ be the probability measure on $h^s$, induced by the map
\begin{equation}\label{mapp}
\omega\longmapsto \big(\zeta^{*}_n\, g_{n}(\omega)\big)_{n\in\N}.
\end{equation}
More precisely, thanks to the assumption \eqref{moment2}, we have that 
$$
\Big(\pi_N\big(\zeta^{*}_n\, g_{n}(\omega)\big)_{n\in\N}\Big)_{N\in\N}
$$
is a Cauchy sequence in $L^2(\Omega;h^s)$
and therefore \eqref{mapp} defines a measurable map from  $(\Omega,{\mathcal F})$ to $(h^s,{\mathcal B})$, where ${\mathcal B}$ is the Borel $\sigma$-algebra of $h^s$.
The measure $\mu$ depends on the sequence  $(\zeta^{*}_n)_{n\in\N}$ and the choice of the measure $\theta$. This (important) dependence will not be explicitly mentioned. 
\\

We also define $\mu_N$ as the probability measure on the finite dimensional space $E_N$, induced by the map
\begin{equation*}
\omega\longmapsto \big(\zeta^{*}_1\, g_{1}(\omega),\cdots,\zeta^{*}_N\, g_{N}(\omega),0,0,\cdots \big).
\end{equation*}
Observe that if $A$ is a Borel set of $h^s$ and $\chi_{A}$ the characteristic function of $A$ then
$$
\int_{E_N} \chi_{A}(\zeta)\mu_N(d\zeta)=\mu_{N}(A\cap E_N)\,.
$$
Indeed, it suffices to use that $ \chi_{A}(\zeta) \chi_{E_N}(\zeta)= \chi_{A\cap E_N}(\zeta)$.
Similarly, we define $\mu^N$ as the probability measure on the orthogonal complementary of $E_N$ in $h^s$, induced by the map
\begin{equation*}
\omega\longmapsto \big(0,\cdots,0,\zeta^{*}_{N+1}\, g_{N+1}(\omega),\zeta^{*}_{N+2}\, g_{N+2}(\omega),\cdots \big).
\end{equation*}
We have the following property of the measure $\mu$.
\begin{lemma}\label{lem1}
Let $\mu$ be the measure defined from a sequence $(\zeta^{*}_n)_{n\in\N}\in h^s$ such that $\zeta^{*}_n\neq 0$ for every $n\in\N$. 
Then every set of full $\mu$ measure is dense in $h^s$. 
Moreover, if for some $\sigma>s$,  $(\zeta^{*}_n)_{n\in\N}\notin h^\sigma$ then $\mu(h^\sigma)=0$.  
\end{lemma}
\begin{proof}
Following \cite[Lemma~B1]{BT1}, we first show that  if for some $\sigma>s$,  $(\zeta^{*}_n)_{n\in\N}\notin h^\sigma$ then $\mu(h^\sigma)=0$.  
We have that the event 
$$
\{
\zeta\in h^s\,:\, \|\zeta\|_{h^\sigma}<\infty
\}
$$
belongs to the asymptotic $\sigma$-algebra obtained from suitable independent $\sigma$-algebras because the property $\|\zeta\|_{h^\sigma}<\infty$ depends only on $(1-\Pi_N)\zeta$ for every $N\in\N$ (see e.g. \cite{Stein}).
Therefore by the Kolmogorov zero-one law, we have that  $\mu(h^\sigma)=0$ or $\mu(h^\sigma)=1$.
We suppose that   $\mu(h^\sigma)=1$ and we look for a contradiction.  
If $\mu(h^\sigma)=1$ then $\|\zeta\|_{h^\sigma}<\infty$,  $\mu$ almost surely and by the dominated convergence 
\begin{equation}\label{zero-one}
\lim_{N\rightarrow\infty} \, 
\int_{h^s}\,e^{-\|\pi_N \zeta\|_{h^\sigma}^2}\,\mu(d\zeta)=\int_{h^s}\, e^{-\|\zeta\|_{h^\sigma}^2}\,\mu(d\zeta)>0\,.
\end{equation}
We will show that 
\begin{equation}\label{diveregence}
\lim_{N\rightarrow\infty} \, \int_{h^s}\, e^{-\|\pi_N \zeta\|_{h^\sigma}^2}\,\mu(d\zeta)=0
\end{equation}
which will be in a contradiction with \eqref{zero-one}.  Using the independence, we can write 
\begin{equation}\label{diveregence_bis}
\int_{h^s}\, e^{-\|\pi_N \zeta\|_{h^\sigma}^2}\,\mu(d\zeta)=\prod_{n=1}^N \int_{\R^2}e^{- |\zeta_n^\star|^{2}(x^2+y^2)n^{2\sigma}}f(x,y)dxdy\,.
\end{equation}
Now,  if we set 
$$
\theta:=\int_{x^2+y^2\leq 1}\, f(x,y)dxdy\in (0,1)
$$
then we can write 
\begin{multline*}
\int_{\R^2}e^{- |\zeta_n^\star|^{2}(x^2+y^2)n^{2\sigma}}f(x,y) dxdy
\leq 
\theta
+
\int_{x^2+y^2> 1}e^{- |\zeta_n^\star|^{2}(x^2+y^2)n^{2\sigma}}f(x,y)dxdy
\\
\leq 
\theta+e^{- |\zeta_n^\star|^{2}n^{2\sigma}}(1-\theta)
=1-(1-\theta)(1-e^{- |\zeta_n^\star|^{2} n^{2\sigma}})\,.
\end{multline*}
Now, we observe that 
$$
\lim_{N\rightarrow\infty} \, \sum_{n=1}^N(1-e^{- |\zeta_n^\star|^{2}n^{2\sigma}})=\infty
$$
because 
$$
\lim_{N\rightarrow\infty} \, 
\sum_{n=1}^N |\zeta_n^\star|^{2}\, n^{2\sigma} =\infty
$$
by assumption. Therefore, we have that \eqref{diveregence} holds because the product defined in \eqref{diveregence_bis} diverges to zero.
\\

Next, we show that every set of full $\mu$ measure is dense in $h^s$. 
For that purpose, we follow  \cite[Proposition~1.2]{BT2}.
We need to show that for every $\alpha=(\alpha_n)_{n\in\N}\in h^s$ and every $\varepsilon>0$
$$
\mu(\zeta\,:\, \|\zeta-\alpha\|_{h^s}<\varepsilon)>0\,.
$$
Using the independence and the triangle inequality, we can write 
$$
\mu(\zeta\,:\, \|\zeta-\alpha\|_{h^s}<\varepsilon)\geq \mu\big(\zeta\,:\, \|\pi_N(\zeta-\alpha)\|_{h^s}<\varepsilon/2\big)\times
\mu\big(\zeta\,:\, \|\pi_N^\perp(\zeta-\alpha)\|_{h^s}<\varepsilon/2\big). 
$$
Using that $\zeta^{*}_n\neq 0$ for every $n\in\N$ and using the support property of the function $f$ defining the measure $\theta$ in \eqref{theta}, we infer that for every $N\geq 1$,
$$
\mu(\zeta\,:\, \|\pi_N(\zeta-\alpha)\|_{h^s}<\varepsilon)>0\,.
$$
Therefore our aim is to bound from below 
$
\mu(\zeta\,:\, \|\pi_N^\perp(\zeta-\alpha)\|_{h^s}<\varepsilon/2),
$
for $N\gg 1$. For that purpose, we can write that if $ \|\pi_N^\perp \alpha\|_{h^s}<\varepsilon/4$ then
$$
\mu(\zeta\,:\, \|\pi_N^\perp(\zeta-\alpha)\|_{h^s}<\varepsilon/2)
\geq
\mu\big(\zeta\,:\, \|\pi_N^\perp \zeta\|_{h^s}<\varepsilon/4
\big).
$$
Recall that $\alpha\in h^s$ is fixed. Hence  for $N\gg 1$,  we have $\|\pi_N^\perp \alpha\|_{h^s}<\varepsilon/4$ and therefore for $N\gg 1$,
 $$
\mu(\zeta\,:\, \|\pi_N^\perp(\zeta-\alpha)\|_{h^s}<\varepsilon/2)
\geq
\mu\big(\zeta\,:\, \|\pi_N^\perp \zeta\|_{h^s}<\varepsilon/4
\big)
=
1-
\mu\big(\zeta\,:\, \|\pi_N^\perp \zeta\|_{h^s}\geq \varepsilon/4
\big).
$$
By the Markov inequality and  thanks to \eqref{moment2}, we can write 
$$
\mu\big(\zeta\,:\, \|\pi_N^\perp \zeta\|_{h^s}\geq \varepsilon/4
\big)\leq
\frac{4}{\varepsilon}
\| \|\pi_N^\perp \zeta\|_{h^s}\|_{L^2(d\mu(\zeta))}^2
\leq \frac{C}{\varepsilon}
\sum_{n>N}|\zeta_n^{*}|^2 n^{2s}\,.
$$
Therefore for $N\gg 1$,
$$
\mu\big(\zeta\,:\, \|\pi_N^\perp \zeta\|_{h^s}\geq \varepsilon/4
\big)<\frac{1}{2}\,.
$$
In summary, for $N\gg 1$
$$
\mu(\zeta\,:\, \|\pi_N^\perp(\zeta-\alpha)\|_{h^s}<\varepsilon/2)>\frac{1}{2}\,.
$$
This completes the proof of Lemma~\ref{lem1}. 
\end{proof}
\subsection{Invariance of the measures} 
Using \eqref{converge}, by the Lebesgue dominated convergence theorem, for every $F\in C_b(h^s;\R)$,
\begin{equation}\label{leoleo} 
\lim_{N\rightarrow \infty}\, \int_{h^s} F(\Phi_N(t)(\pi_N \zeta))\mu(d\zeta)=\int_{h^s} F(\Phi(t)( \zeta))\mu(d\zeta)\,.
\end{equation}
We have the following basic lemma. 
\begin{lemma}\label{lille}
For every $F\in C_b(h^s;\R)$,
\begin{equation}\label{FF}
\int_{h^s} F(\zeta) \mu(d\zeta)=\lim_{N\rightarrow \infty}\int_{E_N}F(\zeta)\mu_N(d\zeta)\,.
\end{equation}
\end{lemma}
\begin{proof}
We follow  \cite[Theorem~1.2]{Tz_f}. 
Let $U$ be an open of $h^s$.  Set
$
U_N:=\{\zeta\in h^s\,:\, \pi_N(\zeta)\in U\}.
$
Then 
\begin{equation}\label{subset}
U\subset \liminf_{N}(U_N),
\end{equation}
where 
$$
\liminf_{N}(U_N):=\bigcup_{N=1}^\infty \bigcap_{N_1=N}^\infty U_{N_1}\,.
$$
Indeed for every $u\in U$, we have that $\pi_N(u)$ converges to $u$ in $h^s$ and hence since $U$ is an open set there is $N_0$ such that for every $N\geq N_0$, we have $\pi_N(u)\in U$, i.e. $u\in U_N$. This is precisely the condition to assure that $u\in \liminf_{N}(U_N)$. 
Denote by $\chi_{U}$ the characteristic function of $U$.
Then using the Fubini theorem, we can write 
$$
\int_{h^s}\chi_{U_N}(\zeta)\mu(d\zeta)  = \int_{E_N}
\Big(
\int_{E_N^\perp}\chi_{U_N}(\zeta)\mu^N(d\zeta^N)\Big)
\mu_N(d\zeta_N)\,,
$$
where $\zeta=(\zeta_N,\zeta^N)$. We have that 
$
\chi_{U_N}(\zeta)=\chi_{U}(\pi_N\zeta)
$ 
and therefore we get
$$
\int_{h^s}\chi_{U_N}(\zeta)\mu(d\zeta)  =  \int_{E_N}\chi_{U}(\zeta)\mu_N(d\zeta)\,.
$$
Using the Fatou lemma and \eqref{subset}, we obtain that 
\begin{equation}\label{open_15_04}
\liminf_{N\rightarrow\infty} \int_{E_N}\chi_{U}(\zeta)\mu_N(d\zeta)\geq \int_{h^s}\chi_{U}(\zeta)\mu(d\zeta)\,.
\end{equation}
By passing to complementary sets, we obtain that for every closed set $F$, 
\begin{equation}\label{closed_15_04}
\limsup_{N\rightarrow\infty} \int_{E_N}\chi_{F}(\zeta)\mu_N(d\zeta)\leq \int_{h^s}\chi_{F}(\zeta)\mu(d\zeta)\,.
\end{equation}
As a consequence, if $A$ is a measurable set such that $\mu(\partial A)=0$ then, we have  
\begin{equation}\label{limm}
\lim_{N\rightarrow\infty} \int_{E_N}\chi_{A}(\zeta)\mu_N(d\zeta)=\int_{h^s}\chi_{A}(\zeta)\mu(d\zeta)=\mu(A) \,.
\end{equation}
Indeed,  using \eqref{open_15_04} and \eqref{closed_15_04}, we can write
\begin{multline*}
 \int_{h^s}\chi_{A}(\zeta)\mu(d\zeta)=  \int_{h^s}\chi_{\mathring{A}}(\zeta)\mu(d\zeta)\leq 
 \liminf_{N\rightarrow\infty} \int_{E_N}\chi_{\mathring{A}}(\zeta)\mu_N(d\zeta)
 \leq 
 \\
 \liminf_{N\rightarrow\infty} \int_{E_N}\chi_{A}(\zeta)\mu_N(d\zeta)
 \leq 
  \limsup_{N\rightarrow\infty} \int_{E_N}\chi_{A}(\zeta)\mu_N(d\zeta)
\leq
\\
 \limsup_{N\rightarrow\infty} \int_{E_N}\chi_{\overline{A}}(\zeta)\mu_N(d\zeta)\leq
  \int_{h^s}\chi_{\overline{A}}(\zeta)\mu(d\zeta)=
   \int_{h^s}\chi_{A}(\zeta)\mu(d\zeta)\,.
\end{multline*}
Therefore we have \eqref{limm}.
Let us now turn to the proof of \eqref{FF}.
By writing 
$$
F=\frac{1}{2}(F+|F|)+\frac{1}{2}(F-|F|),
$$
we conclude that we can suppose that in \eqref{FF}, we have that $F\geq 0$. Let $R>0$ be such that $0\leq F(\zeta)\leq R$. 
Therefore, we can write 
$$
\int_{h^s} F(\zeta) \mu(d\zeta)=\int_{0}^R \mu(A_{\lambda})d\lambda\,,
$$
where 
$
A_{\lambda}=\{\zeta \in h^s \, :\, F(\zeta)\geq \lambda \}. 
$
Since $F$ is continuous, we have that $A_{\lambda}$ is closed and 
$
\{\zeta \in h^s \, :\, F(\zeta)> \lambda \}
$
is open. Therefore 
$
\partial A_{\lambda}\subset\{\zeta \in h^s \, :\, F(\zeta)= \lambda \}. 
$ 
In particular, if $\lambda_1\neq \lambda_2$ then 
$
\partial A_{\lambda_1}
$ 
and 
$
\partial A_{\lambda_2}
$ 
are disjoint.  We claim that 
\begin{equation}\label{bord}
\mu(\partial A_{\lambda})=0, \quad \lambda-{\rm almost~surely}\,.
\end{equation}
Indeed, since $\mu$ is a probability measure, we have that for every $\varepsilon>0$ there can only be a finite number of $\lambda\in [0,R]$ such that 
$\mu(\partial A_{\lambda})>\varepsilon$ (not more than $1/\varepsilon$).  Taking a union on a sequence of  $\varepsilon$ tending to $0$, 
we obtain that there may be only a countable set of $\lambda$'s such that  $\mu(\partial A_{\lambda})>0$. This implies \eqref{bord}.  
Using  \eqref{limm}, \eqref{bord} and the dominated convergence, we obtain that  
$$
\int_{h^s} F(\zeta) \mu(d\zeta)=\int_{0}^R \mu(A_{\lambda})d\lambda
=
\lim_{N\rightarrow\infty}
\int_{0}^R 
\int_{E_N}
\chi_{A_{\lambda}}(\zeta)\mu_N(d\zeta)d\lambda\,.
$$
Now, we observe that 
$$
\int_{E_N}\chi_{A_{\lambda}}(\zeta)\mu_N(d\zeta)=\mu_N(A_\lambda^N),
$$
where 
$
A^N_{\lambda}=\{\zeta \in E_N \, :\, F(\zeta)\geq \lambda \}
$
and therefore 
$$
\int_{0}^R \int_{E_N}\chi_{A_{\lambda}}(\zeta)\mu_N(d\zeta)d\lambda=\int_{E_N}F(\zeta) \mu_N (d\zeta)\,.
$$
This completes the proof of Lemma~\ref{lille}. 
\end{proof}
Using the invariance under rotations of the density of the measure $\theta$ defined in \eqref{theta} and the Liouville theorem for the Hamiltonian flow $\Phi_N(t)$ on $E_N$ (see \eqref{Hamilton}), one has the following statement.
\begin{lemma}\label{div}
For every $F\in C_b(h^s;\R)$,
$$
\int_{E_N} F(\Phi_N(t)(\zeta))\mu_N(d\zeta)=\int_{E_N} F(\zeta) \mu_N(d\zeta)\,.
$$
\end{lemma}
\begin{proof}
Set 
$$
\tilde{F}(\zeta_1,\cdots,\zeta_N):= F(\zeta_1,\cdots,\zeta_N,0,0, \cdots)\,.
$$
We have that 
$$
\pi_N\Phi_N(t)(\zeta)=\Big(\zeta_n e^{it\beta_{N,n}(\zeta)}\Big)_{1\leq n\leq N},\quad \zeta\in E_N\,,
$$
where
$$
\beta_{N,n}(\zeta)=n^2-2\sum_{k=1}^N \min(n,k) |\zeta_k|^2.
$$
We are reduced to compute 
\begin{equation}\label{L1}
\int_{\C^N}
\tilde{F}\big(\zeta_1,\cdots,\zeta_N\big) \big(\prod_{n=1}^N f(\zeta_n)\big) d\zeta_1 \cdots d\zeta_N\,.
\end{equation}
Thanks to our assumption on $f$,
$$
f(\zeta_n)=f\Big(\zeta_n e^{it\beta_{N,n}(\zeta_1,\cdots,\zeta_N)}\Big),\quad 1\leq n\leq N\,.
$$
Moreover, thanks to the Hamiltonian structure \eqref{Hamilton},  using the Liouville theorem, we have that the volume element  $d\zeta_1 \cdots d\zeta_N$ is invariant under the map
$$
\zeta_n\longmapsto \zeta_n e^{it\beta_{N,n}(\zeta_1,\cdots,\zeta_N)}\,,\quad 1\leq n\leq N\,.
$$
Therefore \eqref{L1} equals 
\begin{equation*}
\int_{\C^N}
\tilde{F}\big(\zeta_1\, e^{it\beta_{N,1}(\zeta_1,\cdots,\zeta_N)} ,\cdots,\zeta_N\,e^{it\beta_{N,N}(\zeta_1,\cdots,\zeta_N)}   \big) \big(\prod_{n=1}^N f(\zeta_n)\big) d\zeta_1 \cdots d\zeta_N\,.
\end{equation*}
This completes the proof of Lemma~\ref{div}.
\end{proof}
Using the Fubini theorem 
$$
\int_{E_N} F(\Phi_N(t)(\zeta))\mu_N(d\zeta)=
\int_{E_N} F(\Phi_N(t)(\pi_N \zeta)) \mu_N(d\zeta)=
\int_{h^s} F(\Phi_N(t)(\pi_N \zeta))\mu(d\zeta)\,.
$$
Therefore, thanks to  \eqref{leoleo} and Lemma~\ref{div}, 
$$
\int_{h^s} F(\Phi(t)( \zeta))\mu(d\zeta)=\lim_{N\rightarrow \infty}\int_{E_N}F(\zeta)\mu_N(d\zeta)\,.
$$
Using Lemma~\ref{lille}, we arrive at 
$$
\int_{h^s} F(\Phi(t)( \zeta))\mu(d\zeta)=\int_{h^s} F(\zeta) \mu(d\zeta)
$$
which proves the invariance of $\mu$ under $\Phi(t)$, after using a final approximation argument of an arbitrary $L^1(h^s;d\mu)$ 
function by $C_b(h^s;\R)$ functions. Such an approximation is classical by first approximating the indicator functions of the open sets and then using the regularity of the measure $\mu$. 
\\

Summarizing the previous discussion, we get the following statement. 
\begin{theorem}\label{thm1}
Let $\mu$ be a measure on $h^s$, $s\geq 0$, defined from a sequence $(\zeta^{*}_n)_{n\in\N}\in h^s$, $s\geq 0$ and a rotation invariant probability measure on $\C$ as in \eqref{theta}. Then $\mu$ is invariant under the flow $\Phi(t)$ of \eqref{Birk}. 
\end{theorem}
\section{Proof of  Theorem~\ref{thm2} and a conjecture }
Using Theorem~\ref{thm1}, Lemma~\ref{lem1} and \cite{GK,GKP}, we can prove Theorem~\ref{thm2}.
\begin{proof}[Proof of  Theorem~\ref{thm2}]
Let $s>-1/2$. Let us fix a sequence  $(\zeta^{*}_n)_{n\in\N}\in h^{s+\frac{1}{2}}$ such that
$$
\zeta^{*}_n\neq 0, \forall\, n\in\N,\quad (\zeta^{*}_n)_{n\in\N}\notin h^{\sigma},\,\,\ \forall\, \, \sigma>s+\frac{1}{2}.
$$
An example of such a sequence is 
$$
\zeta^{*}_n=\frac{1}{n^{s+1}\log(n+1) }\,\, ,\quad n\geq 1 \,.
$$
Let $\mu$ be the measure on $ h^{s+\frac{1}{2}}$ associated with this sequence and a rotation invariant probability measure on $\C$ as in \eqref{theta}.
\\

Thanks to \cite{GK,GKP} there is a map (the Birkhoff map) 
$$
\phi\colon H^s_0(\T)\longrightarrow  h^{s+\frac{1}{2}}
$$
which is a bi-continuous bijection such that 
$$
\Phi_1(t)=\phi^{-1}\circ \Phi(t)\circ \phi
$$
generates the Benjamin-Ono flow on $H^s_0(\T)$ ($\Phi(t)$ is the flow appearing in Theorem~\ref{thm1}). Moreover,  we also have that $\phi$ is a bijection from $H^\sigma_0(\T)$ to $ h^{\sigma+\frac{1}{2}}$ for every $\sigma\geq s$.
\\

We define the measure $\rho_s$ as 
$$
\rho_s(A)=\mu(\phi(A)),
$$
where $A$ is an arbitrary Borel set of $H^s_0(\T)$. Now, using the definition and  Theorem~\ref{thm1} we can write 
$$
\rho_s(\Phi_1(t)(A))=\mu(\Phi(t)( \phi(A)))=\mu( \phi(A))=\rho_s(A),
$$
where $A$ is again an arbitrary Borel set of $H^s_0(\T)$. Therefore $\rho_s$ is invariant under $\Phi_1(t)$. 
\\

Using Lemma~\ref{lem1}, we can write for $\sigma>s$,
$$
\rho_s(H^\sigma_0(\T))=\mu(\phi(H^\sigma_0(\T)))=\mu(h^{\sigma+\frac{1}{2}})=0.
$$ 
Let $A$ be such that $\rho_s(A)=1$. Then $\mu(\phi(A))=1$. Using  Lemma~\ref{lem1}, we obtain that $\phi(A)$ is dense in 
$ h^{s+\frac{1}{2}}$. Thanks to the continuity and the bijectivity of $\phi^{-1}$, we have that $A=\phi^{-1}(\phi(A))$ is a dense set of $H^s_0(\T)$ 
(a dense set is mapped to a dense set by the continuous bijection $\phi^{-1}$). This completes the proof of  Theorem~\ref{thm2}. 
\end{proof}
Of course, it would be interesting to understand better the measures $\rho_s$ appearing in Theorem~\ref{thm2}. 
They should heavily depend on the choice of the measure in \eqref{theta} and the sequence $(\zeta^{*}_n)_{n\in\N}$ in the definition of $\mu$.
Here is a conjecture we have.
\begin{conjecture}
It seems reasonable to conjecture that if $\mu$ is a suitable gaussian measure on $h^{s+\frac{1}{2}}$ then the corresponding measure $\rho_s$ appearing in Theorem~\ref{thm2} is absolutely continuous with respect to a suitable gaussian measure on $H^s_0(\T)$. For example, we conjecture that the image of the Gibbs measure defined in \cite{Tz_ptrf} by the Birkhoff map defined in \cite{GK,GKP} is the measure on $h^s$, $s<1/2$ defined as the limit as $N\rightarrow \infty$ of the measures 
$$
G_{N}(\zeta)  \mu(d\zeta),
$$
where the measure $\mu$ is constructed from $\theta$ a standard complex gaussian and the sequence 
$$
\zeta^{*}_n=n^{-1},\quad  n\in\N. 
 $$
The density $G_N$ is given by 
$$
G_{N}(\zeta)=\chi\big(\sum_{k=1}^{N}n|\zeta_n|^2-c_N \big)\exp\Big(
\sum_{k=1}^N \Big(\sum_{k_1=k}^N |\zeta_{k_1}|^2\Big)^2
\big),
$$
where $\chi:\R\rightarrow\R$ is a continuous function with a compact support and $c_N\approx \log(N)$ is a renormalisation constant. We conjecture that a similar procedure may lead to the measures considered in \cite{DTV, TV13a,TV13b,TV14}\,.
\end{conjecture}
\section{An almost sure extension of $\Phi(t)$ to $l^4$, after a renormalisation }
In this section we consider  a sequence $(\alpha_n)_{n\in \N}$ such that 
\begin{equation}\label{divergence}
\sum_{n\in \N}|\alpha_n|^2=\infty
\end{equation}
but
$$
\sum_{n\in\N}|\alpha_n|^4<\infty\,.
$$
The sequence 
$$
\alpha_n=\frac{1}{\sqrt{n}}, \quad n\geq 1
$$
is of particular interest because this sequence (together with gaussians) should appear in the analysis of the Benjamin-Ono equation with data distributed according to the white noise on the circle.  We refer to \cite{Oh} for the invariance of the white noise on the circle under the periodic KdV flow. 
\\

Let $(g_n(\omega))_{n\in \N}$ as in the previous sections.  We suppose that the density $f$ also satisfies
\begin{equation}\label{moment4}
\int_{\R^2}(x^4+y^4)f(x,y)dxdy<\infty,\quad \int_{\R^2}(x^2+y^2)f(x,y)dxdy=1. 
\end{equation}
Let $\mu$ be the measure defined by the map
\begin{equation*}
\omega\longmapsto \big(\alpha_n\, g_{n}(\omega)\big)_{n\in\N}.
\end{equation*}
The flow $\Phi(t)$ is not defined $\mu$ almost surely because, in view of \eqref{divergence} and Lemma~\ref{lem1}, one has $\mu(h^0)=0$. However, using the inequality 
$$
\|(\alpha_n)_{n\in \N}\|_{h^{-1}}\leq C(\sum_{n\in\N}|\alpha_n|^4)^{\frac{1}{4}}
$$
we obtain that 
$$
\Big(\pi_N\big(\alpha_n\, g_{n}(\omega)\big)_{n\in\N}\Big)_{N\in\N}
$$
is a Cauchy sequence in $L^4(\Omega;h^{-1})$.
Therefore, we see $\mu$ as a measure on $h^{-1}$ equipped with the Borel sigma algebra ${\mathcal B}$.  Solving  \eqref{Birk} on the support of $\mu$ is not possible because, in view of \eqref{divergence},  the angles diverge almost surely. However, we can make converge the phases after a suitable renormalisation. 
Recall that the flow map of \eqref{Birk_N} is denoted by $\Phi_N(t)$ and is defined as 
$$
\pi_N \Phi_N(t)(\zeta)=\Big(\zeta_n e^{it\beta_{N,n}(\zeta)}\Big)_{1\leq n\leq N},\quad \zeta\in E_N\,,
$$
where
$$
\beta_{N,n}(\zeta)=n^2-2\sum_{k=1}^N \min(n,k) |\zeta_k|^2, \quad \zeta\in E_N.
$$
We have the following statement.
\begin{theorem}\label{klkl}
Let $n\geq 1$. There is $h_n\in L^2(d\mu)$ such that the sequence 
$$
(\beta_{N,n}(\pi_N \zeta)+2nc_{N})_{N\geq 1}
$$
converges $\mu$ almost surely to $h_n(\zeta)$, where the divergent constant $c_N$ is defined by
$$
c_N=\sum_{k=1}^N |\alpha_k|^2\,.
$$ 
As a consequence, for every $n\in\N$  the sequence 
$$
\big(e^{i2t nc_N} (\Phi_N(t)(\pi_N \zeta))_{n}\big)_{N\geq 1}
$$
converges $\mu$ almost surely to
$
\zeta_n e^{it h_n(\zeta)}  
$
and moreover
$$
\Big(\big(e^{i2t nc_N} (\Phi_N(t)(\pi_N \zeta))_{n}\big)_{n\in\N}\Big)_{N\geq 1}
$$
converges $\mu$ almost surely in $h^{-1}$ to
$$
\Big(\zeta_n e^{it h_n(\zeta)}  
\Big)_{n\in\N}\,.
$$
In addition, the measure $\mu$ is invariant under the map 
\begin{equation}\label{future}
(\zeta_n)_{n\in\N}\longmapsto \Big(\zeta_n e^{it h_n(\zeta)}\Big)_{n\in\N}\,.
\end{equation}
\end{theorem}
\begin{remark}\label{rkrkrk}
If one succeeds to extend the  Birkhoff map to the support of the measure $\mu$ considered in Theorem \ref{klkl},  then one will obtain a renormalized Benjamin-Ono flow as the conjugate of the flow \eqref{future}, for many initial data of Sobolev regularity below $-1/2$.
We believe that to obtain such an extension of the  Birkhoff map, we have to make appeal to unavoidable probabilistic arguments.  
As a consequence of such an extension of the  Birkhoff map, one will obtain the probabilistic well-posedness of the Benjamin-Ono equation in a super-critical regularity regime. More precisely, the proof of Theorem~\ref{klkl} suggests the probabilistic well-posedenss of the  Benjamin-Ono equation in the Fourier-Lebesgue spaces ${\mathcal F}L^{-\frac{1}{2},4}$ (i.e. data $u_0$ such that $\langle n\rangle ^{-1/2}\widehat{u_0}(n)\in l^4(\Z)$). Therefore we believe that the probabilistic well-posedness theory developed in the last years for many dispersive models can be extended to the Benjamin-Ono equation with data of super-critical regularity. 
\end{remark}
\begin{proof}[Proof of Theorem \ref{klkl}]
We will use a renormalisation argument as for example in \cite{Tz}.  We can write
$$
\beta_{N,n}(\pi_N \zeta)+2nc_N=
n^2-2\sum_{k=1}^n (k-n)|\zeta_k|^2-2n\sum_{k=1}^N( |\zeta_k|^2-|\alpha_k|^2)\,.
$$
Therefore, we are reduced  to show that
\begin{equation}\label{Kolm}
\Big(\sum_{k=1}^N( |\zeta_k|^2-|\alpha_k|^2)\Big)_{N\geq 1}
\end{equation}
converges $\mu$ almost surely. 
For that purpose we aim to apply the Kolmogorov almost sure convergence theorem in the probability space $(h^{-1}, {\mathcal B},\mu)$. 
We have that $(|\zeta_k|^2-|\alpha_k|^2)_{k\in\N}$ is a family of independent random variables in $(h^{-1}, {\mathcal B},\mu)$. 
Moreover,  thanks to \eqref{moment4}, we have that 
$$
\int_{h^{-1}}(|\zeta_k|^2-|\alpha_k|^2)\mu(d\zeta)=\int_{\Omega} (|g_k(\omega)|^2-1)dp(\omega)=0\,.
$$
In order to apply the Kolmogorov theorem, we need to show that \eqref{Kolm} converges in $L^2(d\mu)$. Let us show that \eqref{Kolm} is
 a Cauchy sequence in $L^2(d\mu)$. For that purpose, by definition, we can write for $N<M$
$$
\Big\|\sum_{k=N+1}^M( |\zeta_k|^2-|\alpha_k|^2)\Big\|_{L^2(d\mu(\zeta))}=
\Big\|\sum_{k=N+1}^M|\alpha_k|^2(|g_k(\omega)|^2-1)\Big\|_{L^2(\Omega)}\,.
$$
Therefore, using the independence and \eqref{moment4}, we obtain that there is a finite constant $C$ such that 
$$
\Big\|\sum_{k=N+1}^M|\alpha_k|^2(|g_k(\omega)|^2-1)\Big\|^2_{L^2(\Omega)}
=
C\, \sum_{k=N+1}^M|\alpha_k|^4
$$
which tends to zero as $N\rightarrow\infty$.  This proves the $\mu$ almost sure convergence of the sequence 
$$
(\beta_{N,n}( \pi_N \zeta)+2nc_{N})_{N\geq 1}.
$$
The proof of the invariance can be done as in the previous section. The only difference is that Lemma~\ref{div} should be replaced by the following statement. 
\begin{lemma}\label{div_bis}
Let $\tilde{\Phi}_{N}(t)$ be defined as 
$$
(\tilde{\Phi}_{N}(t)(\zeta))_n:=e^{i2t nc_N} (\Phi_N(t)(\zeta))_{n},\quad \zeta\in E_N\,.
$$
Then for every $F\in C_b(h^s;\R)$,
$$
\int_{E_N} F(\tilde{\Phi}_N(t)(\zeta))\mu_N(d\zeta)=\int_{E_N} F(\zeta) \mu_N(d\zeta)\,.
$$
\end{lemma}
\begin{proof}
Again we set  
$
\tilde{F}(\zeta_1,\cdots,\zeta_N):= F(\zeta_1,\cdots,\zeta_N,0,0, \cdots)
$
and we study 
\begin{equation}\label{L1_bis}
\int_{\C^N}
\tilde{F}\big(\zeta_1,\cdots,\zeta_N\big) \big(\prod_{n=1}^N f(\zeta_n)\big) d\zeta_1 \cdots d\zeta_N\,.
\end{equation}
First, we make the variable change 
$$
\zeta_n\longmapsto \zeta_n e^{2it n c_N}\,,\quad 1\leq n\leq N
$$
which leaves the volume element and $f(\zeta_n)$ unchanged. Therefore \eqref{L1_bis} equals 
\begin{equation}\label{L2_bis}
\int_{\C^N}
\tilde{F}\big(\zeta_1\, e^{2it c_N},\cdots,\zeta_N\, e^{2it N c_N}\big) \big(\prod_{n=1}^N f(\zeta_n)\big) d\zeta_1 \cdots d\zeta_N\,.
\end{equation}
Next, we make the variable change 
$$
\zeta_n\longmapsto \zeta_n e^{it\beta_{N,n}(\zeta_1,\cdots,\zeta_N)}\,,\quad 1\leq n\leq N
$$
and as in the proof of Lemma~\ref{div} we obtain that \eqref{L2_bis} equals 
\begin{equation*}
\int_{\C^N}
\tilde{F}\big(\zeta_1\, e^{2it c_N}e^{it\beta_{N,1}(\zeta_1,\cdots,\zeta_N)} ,\cdots,\zeta_N\,e^{2it N c_N}e^{it\beta_{N,N}(\zeta_1,\cdots,\zeta_N)}   \big) \big(\prod_{n=1}^N f(\zeta_n)\big) d\zeta_1 \cdots d\zeta_N\,.
\end{equation*}
This completes the proof of Lemma~\ref{div_bis}.
\end{proof}
This completes the proof of Theorem~\ref{klkl}.
\end{proof}
{\bf{Acknowledgements:}}
I am grateful to Patrick G\' erard for several discussions which motivated this work.  
This work is partially supported by the ANR project Smooth ANR-22-CE40-0017. 


\begin{thebibliography}{10}
%
\bibitem{BT1} N.~Burq, N.~Tzvetkov,  {\it Random data Cauchy theory for supercritical wave equations. I. Local theory}, Invent. Math. 173 (2008), no. 3, 449--475.
%
\bibitem{BT2} N.~Burq, N.~Tzvetkov,  {\it Probabilistic well-posedness for the cubic wave equation}, J. Eur. Math. Soc. (JEMS) 16 (2014), no. 1, 1--30.
%
\bibitem{Deng} Y.~Deng,  {\it Invariance of the Gibbs measure for the Benjamin-Ono equation},  J. Eur. Math. Soc. (JEMS) 17 (2015), no. 5, pp. 1107--1198.
%
\bibitem{DTV}  Y.~Deng,  N.~Tzvetkov, N.~Visciglia, {\it Invariant measures and long time behaviour for the Benjamin-Ono equation III},  Comm. Math. Phys. 339 (2015), 815--857. 

\bibitem{GK} P. G\'erard, T. Kappeler,  {\it On the integrability of the Benjamin-Ono equation on the torus}, Comm. Pure Appl. Math. 74 (2021), no. 8, 1685--1747.

\bibitem{GKP} P. G\'erard, T. Kappeler, P. Topalov, {\it Sharp well-posedness results of the Benjamin-Ono equation in $H^s(\T;\R)$ and qualitative properties of its solution}, arXiv:2004.04857 [math.AP].
 
 \bibitem{monica}  R. Killip, T. Laurens, M.Visan, {\it Sharp well-posedness for the Benjamin--Ono equation}, arXiv:2304.00124 [math.AP].
 
\bibitem{Oh} T. Oh,  {\it Invariance of the white noise for KdV}, Comm. Math. Phys. 292 (2009), no. 1, 217--236.
 
 \bibitem{Stein} E.~Stein, R.~Shakarchi, {\it  Real analysis. Measure theory, integration, and Hilbert spaces},  Princeton Lectures in Analysis, 3, Princeton University Press, Princeton, NJ, 2005. 
 
\bibitem{Tz_ptrf} N.~Tzvetkov, {\em Construction of a Gibbs measure associated to the periodic Benjamin-Ono equation}, PTRF 146 (2010), 481-514.
 
 \bibitem{Tz_f} N.~Tzvetkov,  {\it Invariant measures for the defocusing nonlinear Schr\"odinger equation},  Ann. Inst. Fourier (Grenoble) 58 (2008) 2543--2604.
 
\bibitem{Tz}  N.~Tzvetkov,  {\it EDP non lin\'eaires en pr\'esence d'al\'ea singulier}, Gaz. Math.160 (2019) 6--14.

\bibitem{TV13a} N. Tzvetkov, N. Visciglia,  {\it Gaussian measures associated to the higher order conservation laws of the Benjamin-Ono equation},  Ann. Sci. ENS 46 (2013) 249--299.

\bibitem{TV13b}  N.~Tzvetkov,  N.~Visciglia, {\it Invariant measures and long time behaviour for the Benjamin-Ono equation}, Int. Math. Res. Not. 17 (2014) 4679--4614.
%
\bibitem{TV14} N.~Tzvetkov,  N.~Visciglia, {\it  Invariant measures and long time behaviour for the Benjamin-Ono equation II},  J. Math. Pures Appl. 103 (2015), 102--141.
 
 
\end{thebibliography}
\end{document}